\documentclass[12pt]{article}
\usepackage{a4wide}
\usepackage{amsthm}
\usepackage{amsfonts}
\usepackage{amssymb}
\usepackage{amsmath}
\usepackage{cite}
\usepackage{epsfig}
\usepackage{stmaryrd}
\newtheorem{theorem}{Theorem}
\newtheorem{lemma}[theorem]{Lemma}
\newtheorem{proposition}[theorem]{Proposition}
\newcommand\TT{{\cal T}}
\newcommand\GG{{\cal G}}
\newcommand\JJ{{\cal J}}
\newcommand\NN{{\mathbb N}}

\newcommand\dd{\mbox{d}}
\DeclareMathOperator{\sgn}{sgn}
\usepackage{mathtools}

\begin{document}
\title{The dimension of the region of feasible tournament profiles}

\author{Daniel Kr{\'a}l'\thanks{Faculty of Informatics, Masaryk University, Botanick\'a 68A, 602 00 Brno, Czech Republic. E-mail: {\tt \{dkral,540987\}@fi.muni.cz}. Supported by the MUNI Award in Science and Humanities (MUNI/I/1677/2018) of the Grant Agency of Masaryk University.}\and
        \newcounter{lth}
	\setcounter{lth}{1}
	Ander Lamaison\thanks{Institute for Basic Science, 55 Expo-ro, Yuseong-gu, 34126 Daejeon, South Korea. E-mail: {\tt ander@ibs.re.kr}. Previous affiliation: Faculty of Informatics, Masaryk University, Botanick\'a 68A, 602 00 Brno, Czech Republic. This author was also supported by the MUNI Award in Science and Humanities (MUNI/I/1677/2018).}\and
	Magdalena Prorok\thanks{AGH University of Krakow, al.~Mickiewicza 30, 30-059 Krak\'ow, Poland. E-mail: {\tt prorok@agh.edu.pl}.}\and
	Xichao Shu$^\fnsymbol{lth}$}

\date{}

\maketitle

\begin{abstract}
Erd\H os, Lov\'asz and Spencer showed in the late 1970s that
the dimension of the region of $k$-vertex graph profiles,
i.e., the region of feasible densities of $k$-vertex graphs in large graphs,
is equal to the number of non-trivial connected graphs with at most $k$ vertices.
We determine the dimension of the region of $k$-vertex tournament profiles.
Our result, which explores an interesting connection to Lyndon words,
yields that the dimension is much larger than just the number of strongly connected tournaments,
which would be the answer expected as the analogy to the setting of graphs.
\end{abstract}

\section{Introduction}
\label{sec:intro}

Understanding the interplay between substructures of combinatorial objects
underlies a vast number of problems in extremal combinatorics.
One of many examples of this phenomenon is
the famous Erd\H os-Rademacher problem,
which concerns determining the minimum number of triangles in a graph with a given edge density.
It took more than five decades until this challenging problem was fully resolved by Razborov~\cite{Raz08} as
one of the first applications of his flag algebra method introduced in~\cite{Raz07};
also see~\cite{LiuPS17,PikR17,Rei16} for additional related results. 
The resolution of the Erd\H os-Rademacher problem
determined the region of feasible densities of $K_2$ and $K_3$ in large graphs (profiles of $K_2$ and $K_3$).
In this paper, we are interested in the following basic question:
what is the dimension of the region of feasible substructure densities?
Our main result is determining the dimension of the region of $k$-vertex tournament profiles.

The dimension of the region of $k$-vertex graph profiles,
i.e., the region of feasible densities of $k$-vertex graphs in large graphs,
was determined by Erd\H os, Lov\'asz and Spencer~\cite{ErdLS79} in the late 1970s.
They showed that the dimension of this region
is equal to the number of non-trivial connected graphs with at most $k$ vertices (when we say that the region has dimension $d$,
we mean that the region is a $d$-dimensional manifold with boundary).
We next state their result formally using the language of combinatorial limits (we refer
the reader to Subsection~\ref{subsec:lim} for definitions as needed),
which we use throughout the paper;
we write $\GG_k$ for the set of all graphs with at most $k$ vertices,
$\GG^C_k$ for the set of all non-trivial connected graphs with at most $k$ vertices,
i.e., connected graphs with at least two and at most $k$ vertices, and
$B_{\varepsilon}(x)$ for the $\varepsilon$-ball (in the Euclidean metric) with the center at a point $x$.
We remark that we use $v_{z\in Z}$ as a shorthand notation for a vector indexed by elements of $Z$.
\begin{theorem}[{Erd\H os, Lov\'asz and Spencer~\cite{ErdLS79}}]
\label{thm:graph}
For every $k\ge 2$,
there exist $x_0\in [0,1]^{\GG^C_k}$ and $\varepsilon>0$ such that
for every $x\in B_{\varepsilon}(x_0)\subseteq [0,1]^{\GG^C_k}$,
there exists a graphon $W$ such that \[t(G,W)_{G\in\GG^C_k}=x,\]
i.e., $t(G,W)=x_G$ for every $G\in\GG^C_k$.
Moreover,
there exists a polynomial function $f:[0,1]^{\GG^C_k}\to [0,1]^{\GG_k}$ such that
the following holds for every graphon $W$:
\[f(t(G,W)_{G\in\GG^C_k})=t(G,W)_{G\in\GG_k}.\]
\end{theorem}
\noindent Theorem~\ref{thm:graph} asserts that
the densities of non-trivial connected graphs are independent,
i.e, none of these densities is a function of the remaining ones, and
the density of any other graph is a polynomial function of the densities of connected graphs.
In particular,
the dimension of the region of feasible densities of graphs with at most $k$ vertices
is equal the number of non-trivial connected graphs with $k$ vertices.
Since densities of $k$-vertex graphs determine densities of all graphs with at most $k$ vertices,
the dimensions of the region of feasible densities of \emph{graphs with at most $k$ vertices} and
the region of feasible densities of \emph{graphs with exactly $k$ vertices} ($k$-vertex graph profiles) are the same.
Hence, the dimension of the region of $k$-vertex graph profiles
is indeed equal to the number of non-trivial connected graphs with at most $k$ vertices.

Our main result is determining the dimension of the region of $k$-vertex tournament profiles,
i.e., the region of feasible densities of $k$-vertex tournaments in large tournaments.
In the analogy to the graph case,
it is natural to expect that
the dimension of this region
is equal to the number of non-trivial strongly connected tournaments with at most $k$ vertices.
This is indeed a lower bound on the dimension of the region, however, the actual dimension is larger.
Similarly to the case of permutation profiles studied in~\cite{BorP20,GleHKKKL17,GarKMP23},
there is a close connection to Lyndon words, which appear in various contexts in algebra, combinatorics, and computer science.
We say that a tournament $T$ is \emph{Lyndon}
if its strongly connected components listed in the order given by $T$ form a Lyndon word (a formal definition is given and
properties of Lyndon tournaments are discussed in Subsection~\ref{subsec:lyndon});
in particular, every strongly connected tournament is Lyndon.
Let $\TT_k$ be the set of all tournaments with at most $k$ vertices and
let $\TT^L_k$ the set of all non-trivial Lyndon tournaments with at most $k$ vertices,
i.e., Lyndon tournaments with at least two and at most $k$ vertices.
Our main result, which is implied by Theorems~\ref{thm:upth} and~\ref{thm:lowth}, is the following.
\begin{theorem}
\label{thm:tourn}
For every $k\ge 3$,
there exist $x_0\in [0,1]^{\TT^L_k}$ and $\varepsilon>0$ such that
for every $x\in B_{\varepsilon}(x_0)\subseteq [0,1]^{\TT^L_k}$,
there exists a tournamenton $W$ such that \[t(T,W)_{T\in\TT^L_k}=x.\]
Moreover,
there exists a polynomial function $f:[0,1]^{\TT^L_k}\to [0,1]^{\TT_k}$ such that
the following holds for every tournamenton $W$:
\[f(t(T,W)_{T\in\TT^L_k})=t(T,W)_{T\in\TT_k}.\]
\end{theorem}
\noindent In other words,
the densities of non-trivial Lyndon tournaments are independent, and
the density of any other tournament is a polynomial function of the densities of non-trivial Lyndon tournaments.
In particular, the dimension of the region of feasible densities of tournaments with at most $k$ vertices
is equal to the number of non-trivial Lyndon tournaments with at most $k$ vertices.
Since the densities of tournaments \emph{with exactly $k$ vertices}
determine the densities of tournaments \emph{with at most $k$ vertices},
the dimension of the region of $k$-vertex tournament profiles
is also equal to the number of non-trivial Lyndon tournaments with at most $k$ vertices.

The paper is organized as follows.
We start with reviewing the notation and the concepts used throughout the paper in Section~\ref{sec:prelim}.
In Section~\ref{sec:upper},
we combine flag algebra techniques and properties of Lyndon words to show that the density of any tournament $T$
is a polynomial function of the densities of Lyndon tournaments with the same or smaller number of vertices than that of $T$.
In Section~\ref{sec:lower},
we apply analytic techniques from the theory of combinatorial limits
in conjunction with properties of Lyndon words to show that
the densities of Lyndon tournaments are independent.
This determines the dimension of the region of $k$-vertex tournament profiles for every $k\ge 3$ and
so proves Theorem~\ref{thm:tourn}.

\section{Preliminaries}
\label{sec:prelim}

In this section, we introduce notation used throughout the paper.
We write $[n]$ for the set of the first $n$ positive integers;
we use the shorthand notation $x_{[n]}$ for a sequence $x_1,\ldots,x_n$ (or the vector formed by $x_1,\ldots,x_n$) and
more generally $x_Z$ for a sequence indexed by elements of a set $Z$.
All graphs considered in this paper are simple;
if $G$ is a graph, then $V(G)$ and $E(G)$ are the vertex set and the edge set of $G$, respectively.
The number of vertices of a graph $G$ is denoted by $|G|$.
A \emph{tournament} is an orientation of the complete graph;
we use $V(G)$, $E(G)$ and $|G|$, analogously as in the setting of graphs.
A tournament is \emph{transitive} if it has no directed cycle.
A subtournament of $T$ \emph{induced} by a set $X$ of vertices
is the subtournament formed by the vertices $X$.
The \emph{direct sum} of tournaments $T_1$ and $T_2$, denoted by $T_1\oplus T_2$,
is the tournament obtained from $T_1$ and $T_2$ by adding all edges between $T_1$ and $T_2$ and directing them from $T_1$ and $T_2$.
Observe that a tournament is strongly connected if and only if it is not a direct sum of two tournaments.

\subsection{Combinatorial limits}
\label{subsec:lim}

We now introduce basic concepts from the theory of combinatorial limits,
focusing on graph limits and tournament limits.
We refer to the monograph by Lov\'{a}sz~\cite{Lov12} for a comprehensive introduction to the theory of graph limits.
The \emph{homomorphic density} of a graph $H$ in a graph $G$, denoted by $t(H,G)$,
is the probability that a random function $f: V(H) \rightarrow V(G)$ is a homomorphism,
i.e., a function $f$ such that $f(u)f(v)$ is an edge of $G$ for every $uv\in E(H)$.
We say that a sequence of graphs $(G_n)_{n\in N}$ is \emph{convergent}
if the number of vertices of $G_n$ tends to infinity and
the sequence of homomorphic densities $t(H, G_n)$ converges for every graph $H$.
Convergent sequences of graphs can be represented by an analytic object called a graphon:
a \emph{graphon} is a measurable function $W: [0,1]^2 \rightarrow [0,1]$ that is symmetric,
i.e., $W(x,y) = W(y,x)$ for $(x,y) \in [0,1]^2$.
The definition of a homomorphic density of a graph extends to graphons:
the \emph{homomorphic density} of a graph $H$ in a graphon $W$, denoted by $t(H,W)$, is defined as
\[ t(H,W) = \int_{[0,1]^{V(H)}} \prod_{uv\in E(H)} W(x_u,x_v) \dd x_{V(H)}.\]
We often simply say the \emph{density} of a graph $H$ in a graphon $W$
instead of the homomorphic density $H$ in $W$.
A graphon $W$ is a \emph{limit} of a convergent sequence $(G_n)_{n\in N}$ of graphs
if $t(H, W)$ is the limit of $(t(H,G_n))_{n\in N}$ for every graph $H$.
Lov\'{a}sz and Szegedy in~\cite{LovS06} showed that
every convergent sequence of graphs has a limit graphon and
every graphon is a limit of a convergent sequence of graphs;
also see~\cite{DiaJ08} for a relation to exchangeable arrays.
Conversely, for every graphon $W$,
there is a convergent sequence of graphs such that $W$ is its limit.

We now introduce tournament limits as
studied and applied particularly in the setting of extremal combinatorics in~\cite{ChaGKN19,GrzKLV23,Tho18,ZhaZ20}.
Analogously to the graph case,
the \emph{density} of a tournament $T$ in a tournament $S$, denoted by $t(T,S)$,
is the probability that a random injective function $f:V(T)\rightarrow V(S)$ is a homomorphism, and
a sequence of tournaments $(S_n)_{n\in\NN}$ is \emph{convergent}
if the number of vertices of $S_n$ tends to infinity and
the sequence of densities $t(T,S_n)_{n\in\NN}$ converges for every tournament $T$.
A \emph{tournamenton} is a measurable function $W: [0,1]^2 \rightarrow [0,1]$ satisfying that
$W(x,y) + W(y,x) = 1$ for all $(x,y) \in [0,1]^2$.
The \emph{density} of a tournament $T$ in a tournamenton $W$, denoted by $t(T,W)$, is
\begin{equation}
t(T,W) = \int_{[0,1]^{V(T)}} \prod_{\overrightarrow{uv}\in E(T)} W(x_u,x_v) \dd x_{V(T)};\label{eq:tdef}
\end{equation}
we use $\overrightarrow{uv}$ to emphasize the direction of an edge in a tournament $T$.
A tournamenton is a \emph{limit} of a convergent sequence $(S_n)_{n \in \NN}$ of tournaments
if $t(T, W)$ is the limit of $(t(T,S_n))_{n\in N}$ for every tournament $T$.
As in the case of graphs,
every convergent sequence of tournaments has a limit tournamenton, and
every tournamenton is a limit of a convergent sequence of tournaments.

\subsection{Product of tournaments}
\label{subsec:flag}

The flag algebra method introduced by Razborov in~\cite{Raz07} led to progress on many challenging problems in extremal combinatorics;
the method has been successfully applied particularly in the settings of
graphs~\cite{Raz08,Grz12,HatHKNR12,KraLSWY13,HatHKNR13,PikV13,BabT14,BalHLL14,PikR17,GrzHV19,PikST19},
digraphs~\cite{HlaKN17,CorR17}, hypergraphs~\cite{Raz10,BabT11,GleKV16,BalCL22},
geometric problems~\cite{KraMS12} and permutations~\cite{BalHLPUV15,SliS18,CruDN23}.
In our consideration, the most important is the concept of the product, which we define next.
The same concept also appears in the setting of homomorphisms as discussed in~\cite{Lov12}.

\begin{figure}
\begin{center}
\epsfbox{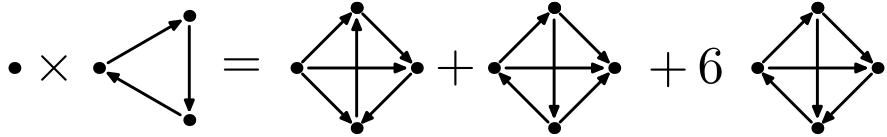}
\end{center}
\caption{An example of the product of two tournaments.}
\label{fig:flag}
\end{figure}

Before defining this notion, we need another notation.
If $\tau$ is a formal linear combinations of tournaments, say $\tau=\alpha_1 T_1+\cdots \alpha_k T_k$,
we write $t(\tau,W)$ for the corresponding linear combination $\alpha_1 t(T_1,W)+\cdots+\alpha_k t(T_k,W)$ of densities in a tournamenton.
If $T_1$ and $T_2$ are two tournaments,
then the (flag) \emph{product} of $T_1$ and $T_2$ denoted by $T_1\times T_2$,
is a formal linear combination of all tournaments $T$ with $|T_1|+|T_2|$ vertices that
can be obtained from $T_1$ and $T_2$ by adding edges between $T_1$ and $T_2$ arbitrarily, and
the coefficient at $T$ in the linear combination
is equal to the number of ways that the edges can be added between $T_1$ and $T_2$ yielding $T$.
See Figure~\ref{fig:flag} for an example.
The definition of the product can be linearly extended to formal linear combinations of tournaments,
e.g., $T_1\times (T_2+2\,T_3)=T_1\times T_2+2\,T_1\times T_3$, and
the product (of formal linear combinations of tournaments) is associative,
i.e., $(T_1\times T_2)\times T_3=T_1 \times (T_2\times T_3)$.

The key property of the product is the next identity,
which holds for any two tournaments $T_1$ and $T_2$ and any tournamenton $W$ and
which follows from~\cite{Raz07}:
\[t(T_1\times T_2,W)=t(T_1,W) \cdot t(T_2,W).\]

\subsection{Lyndon words and Lyndon tournaments}
\label{subsec:lyndon}

Lyndon words, introduced by \v{S}ir\v{s}ov~\cite{Shi53} and Lyndon~\cite{Lyn54} in the 1950s,
have many applications in algebra, combinatorics and computer science.
Let $\Sigma$ be a linearly ordered alphabet, and
let the associated lexicographic order on words over $\Sigma$ be denoted by $\preceq$.
For example, if $\Sigma=\{a,b\}$ with the usual order on letters,
then the words of length two are ordered as $a\preceq aa\preceq ab\preceq b\preceq ba\preceq bb$.
A word $s_1 \cdots s_n$ over the alphabet $\Sigma$ is \emph{Lyndon}
if no proper suffix of the word $s_1 \cdots s_n$ is smaller (in the lexicographic order) than the word $s_1 \cdots s_n$ itself. 
For example,
the words $ab$ and $aab$ are Lyndon but the word $abaabb$ is not
because the word $aabb$ is smaller than the word $abaabb$ itself.

We next survey several well-known properties of Lyndon words.
We start with the following property given by the Chen-Fox-Lyndon Theorem~\cite{CheFL58}.

\begin{proposition}
\label{prop1}
Every word over a linearly ordered alphabet can be uniquely expressed as
a concatenation of lexicographically non-increasing Lyndon words.
\end{proposition}

\noindent For example, the word $ababaab$ is the concatenation of Lyndon words $ab$, $ab$ and $aab$.

A \emph{subword} of a word $s_1 \ldots s_n$ is any word $s_{i_1} \ldots s_{i_m}$ with $1 \leq i_1 < \cdots < i_m \leq n$.
The \emph{shuffle product} of words $s_1 \ldots s_n$ and $t_1 \ldots t_m$,
which is denoted by $s_1 \ldots s_n \otimes_S t_1 \ldots t_m$,
is the formal sum of all $\binom{n+m}{n}$ (not necessarily distinct) words of length $n+m$ that
contain the words $s_1 \ldots s_n$ and $t_1 \ldots t_m$ as subwords formed by disjoint sets of letters.
For example, $ab \otimes_S ac = 2aabc+2aacb+abac+acab$.
The following lemma can be found in~\cite{Rad79}.
 
\begin{proposition}
\label{prop:le1}
Let $w_1, \dots, w_n$ be Lyndon words such that $w_1 \succeq \cdots \succeq w_n$.
The lexicographically largest constituent in the shuffle product $w_1 \otimes_S \cdots \otimes_S w_n$
is the term $w_1\dots w_n$ formed by the concatenation of the words $w_1,\ldots,w_n$ (in this order), and
if the words $w_1,\dots,w_n$ are distinct, then the coefficient of this term is equal to one.
\end{proposition}

In this paper,
the alphabet $\Sigma^T$ is the set of strongly connected tournaments.
As the linear order on $\Sigma^T$, we fix any order such that
strongly connected tournaments with smaller number of vertices precede those with the larger number of vertices, and
strongly connected tournaments with the same number of vertices are ordered arbitrarily.
Hence, the smallest element of $\Sigma^T$ is the single-vertex tournament,
the second smallest is the cyclically oriented triangle, and
the third smallest the unique four-vertex tournament containing a Hamilton cycle.
Note that every tournament $T$ can be uniquely expressed as the direct sum of strongly connected tournaments, and
let $w(T)$ be the word formed by the strongly connected tournaments whose direct sum equals $T$ (the strongly connected tournaments
appear in $w(T)$ in the same order as in the direct sum yielding $T$).
A tournament $T$ is \emph{Lyndon}
if the word $w(T)$ is Lyndon (with respect to the order on $\Sigma^T$).
In particular, all strongly connected tournaments are Lyndon.
Lyndon tournaments with at most five vertices are given in Figure~\ref{fig:Lyndon}.
Recall that
$\TT_k$ is the set of all tournaments with at most $k$ vertices, and
$\TT^L_k$ is the set of all non-trivial Lyndon tournaments with at most $k$ vertices,
i.e., Lyndon tournaments with at least two and at most $k$ vertices.

\begin{figure}
\begin{center}
\epsfbox{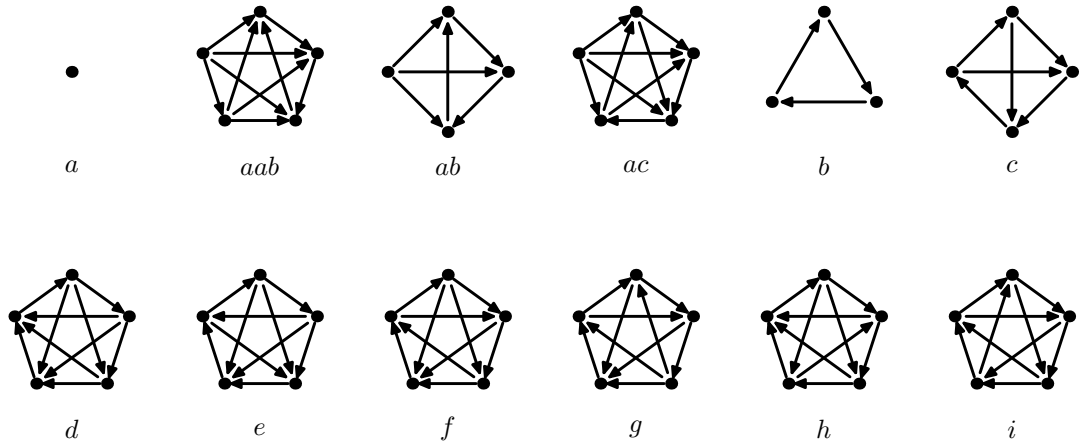}
\end{center}
\caption{Lyndon tournaments with at most five vertices listed in the lexicographic order of the words associated with them.
         The letters $a,\ldots,i$ are used to denote the nine tournaments of $\Sigma^T$ that
	 are the smallest in the chosen linear order
	 when a particular order on the strongly connected $5$-vertex tournaments is fixed.}
\label{fig:Lyndon}
\end{figure}

\section{Upper bound on the dimension}
\label{sec:upper}

In this section,
we show that
the density of every $k$-vertex tournament can be expressed as
a polynomial in the densities of non-trivial Lyndon tournaments with at most $k$ vertices.
This is formally stated in the next theorem.
We slightly abuse the notation when $k=1$ or $k=2$:
the set $\TT^L_k$ is empty in such case and the theorem simply asserts that $p_T$ is constant.

\begin{theorem}
\label{thm:upth}
Let $T$ be a tournament with $k$ vertices.
There exists a polynomial $p_T$ in $|\TT^L_k|$ variables indexed by tournaments $S\in\TT^L_k$ such that
\[ t(T,W)= p_T\left(t(S,W)_{S \in \TT^L_k}\right) \]
for every tournamenton $W$.
\end{theorem}

To prove Theorem~\ref{thm:upth},
we introduce a linear order $<$ on all tournaments using the linearly ordered alphabet $\Sigma^T$.
A tournament $S$ is smaller than a tournament $T$, i.e., $S<T$, if
\begin{itemize}
\item the tournament $S$ has fewer vertices than $T$, or
\item the tournaments $S$ and $T$ have the same number of vertices
      but the number of strongly connected components of $S$ is smaller than that of $T$, or
\item the tournaments $S$ and $T$ have the same number of vertices and
      the same number of strongly connected components
      but the word $w(S)$ is lexicographically smaller than $w(T)$ (with respect to the order on $\Sigma^T$).
\end{itemize}
The eight smallest tournaments in the order $<$ can be found in Figure~\ref{fig:ordT}.

\begin{figure}
\begin{center}
\epsfbox{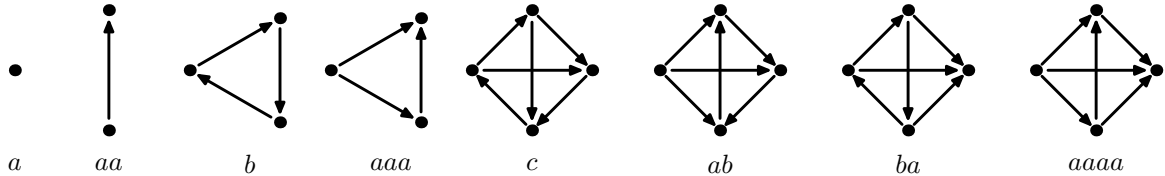}
\end{center}
\caption{The eight smallest tournaments in the order $<_T$.
         The letters describing the strongly connected tournaments
	 are the same as in Figure~\ref{fig:Lyndon}.}
\label{fig:ordT}
\end{figure}

We next state the key lemma to prove Theorem~\ref{thm:upth}.
Recall that the word $w(T)$ associated with a tournament $T$
has a unique decomposition into non-increasing Lyndon words by Proposition~\ref{prop1}.
Hence,
for every tournament $T$,
there exist unique Lyndon tournaments $T_1,\ldots,T_{\ell}$ such that
$T=T_1\oplus\cdots\oplus T_{\ell}$ and $w(T_1)\succeq\cdots\succeq w(T_{\ell})$.
In particular, if $T$ is Lyndon, then $\ell=1$ and $T_1=T$.

\begin{lemma}
\label{lm:le2}
Let $T$ be a tournament with $k$ vertices that is not Lyndon, and
let $T_1,\ldots,T_{\ell}$ be the unique Lyndon tournaments such that
$T=T_1\oplus\cdots\oplus T_{\ell}$ and $w(T_1)\succeq\cdots\succeq w(T_{\ell})$.
There exist a real $\gamma$ and
reals $\alpha_S$ indexed by $k$-vertex tournaments $S$ smaller than $T$ such that
\[t(T,W)=\gamma\prod_{i\in [\ell]}t(T_i,W)+\sum_{|S|=k,\,S<T}\alpha_S\,t(S,W)\]
for every tournamenton $W$.
\end{lemma}

\begin{proof}
For each $k$-vertex tournament $S$,
let $\beta_S$ be the coefficient of $S$ in the flag product of the tournaments $T_1,\ldots,T_\ell$, i.e.,
\[T_1\times\cdots\times T_\ell=\sum_{|S|=k}\beta_S S.\]
Recall that the properties of the product yield that
\[\prod_{i\in [\ell]}t(T_i,W)=\sum_{|S|=k}\beta_S t(S,W)\]
for every tournamenton $W$.
We next analyze which of the coefficients $\beta_S$ are non-zero.
Note that the definition of the product $T_1\times\cdots\times T_\ell$ implies that
$\beta_S$ is non-zero for a $k$-vertex tournament $S$ if and only if
the tournament $S$ can be obtained by taking the tournaments $T_1,\ldots,T_{\ell}$ and adding edges between them.
In particular, $\beta_T\not=0$.

Consider a $k$-vertex tournament $S$ with $\beta_S\not=0$,
i.e., $S$ can be obtained by adding edges between the tournaments $T_1,\ldots,T_{\ell}$.
If adding the new edges merges some strongly connected components of the tournaments $T_1,\ldots,T_{\ell}$,
then the tournament $S$ has fewer strongly connected components than the tournament $T=T_1\oplus\cdots\oplus T_{\ell}$.
Hence, the tournament $S$ is smaller than $T$ in the order $<$.
Assume that adding the new edges does not merge any strongly connected components, i.e.,
the strongly connected components of $S$ correspond one-to-one to the strongly connected components of $T$.
In particular, the word $w(S)$ is a permutation of the letters of the word $w(T)$.
Moreover, for every $T_i$, $i\in [\ell]$,
the strongly connected components of $T_i$ must follow the same order in $S$ as in $T_i$.
It follows that the word $w(S)$ is a constituent in the shuffle product $w(T_1)\otimes_S\cdots\otimes_S w(T_\ell)$.
Since $w(T)$ is the largest constituent in this shuffle product by Proposition~\ref{prop:le1},
we obtain that the tournament $S$ is smaller than $T$ in the order $<$ unless $S=T$
also when adding the new edges does not merge any strongly connected components.

We have shown that $\beta_S$ is non-zero for a $k$-vertex tournament $S\not=T$
only if $S$ is smaller than $T$ in the order $<$.
The statement now follows by setting $\gamma=1/\beta_T$ and $\alpha_S=-\beta_S/\beta_T$.
\end{proof}

We are now ready to prove the main theorem of this section.

\begin{proof}[Proof of Theorem~\ref{thm:upth}]
We proceed by induction using the order $<$ on the tournaments.
The base case is when $T$ is the $1$-vertex tournament and
we set $p_T$ to be the constant $1$.

We now present the induction step.
Let $T$ be a tournament with $k\ge 2$ vertices.
If the tournament $T$ is Lyndon,
we simply set the polynomial $p_T$ as
\[p_T\left(x_{\TT^L_k}\right):=x_T.\]
Otherwise, Lemma~\ref{lm:le2} yields that
there exist a real $\gamma$ and
reals $\alpha_S$ indexed by $k$-vertex tournaments $S$ smaller than $T$ such that
\[t(T,W)=\gamma\prod_{i\in [\ell]}t(T_i,W)+\sum_{|S|=k,\,S<T}\alpha_S\,t(S,W)\]
for every tournamenton $W$,
where $T_1,\ldots,T_{\ell}$ are the unique Lyndon tournaments such that
$T=T_1\oplus\cdots\oplus T_{\ell}$ and $w(T_1)\succeq\cdots\succeq w(T_{\ell})$.
Hence, we can set the polynomial $p_T$ as
\[p_T\left(x_{\TT^L_k}\right):=\gamma\prod_{i\in [\ell], T_i\not=\bullet}x_{T_i}+\sum_{|S|=k,\,S<T}\alpha_S\,p_S\left(x_{\TT^L_k}\right)\]
where the polynomials $p_S$ for $k$-vertex tournaments $S$ smaller than $T$ exist by induction.
We remark that
if $T$ is the transitive tournament on $k$ vertices,
then the product in the definition of the polynomial $p_T$ is empty and so interpreted as equal to one,
i.e.,
\[p_T\left(x_{\TT^L_k}\right):=\gamma+\sum_{|S|=k,\,S<T}\alpha_S\,p_S\left(x_{\TT^L_k}\right)\]
in this case.
In particular, if $T$ is the unique $2$-vertex tournament,
then $p_T$ is the constant equal to $1/2$ (indeed, $\gamma=1/2$ in this case).
\end{proof}

\section{Lower bound on the dimension}
\label{sec:lower}

In this section,
we complete the proof of Theorem~\ref{thm:tourn}
by constructing a family of tournamentons
where the densities of non-trivial Lyndon tournaments with at most $k$ vertices are independent.
Fix $k\ge 3$ for the rest of this section.
Let $\ell=|\TT^L_k|$ and
let $T_1,\ldots,T_{\ell}$ be the tournaments from $\TT^L_k$
listed in the decreasing lexicographic order (with respect to the order on $\Sigma^T$).
For example, when $k=4$, we have $\ell=3$ and
the tournaments $T_1$, $T_2$ and $T_3$ 
are the tournaments associated with the words $c$, $b$ and $ab$ in Figure~\ref{fig:Lyndon}.
Let $n_i$ be the number of vertices of $T_i$, $i\in [\ell]$,
let $v_{i,1},\ldots,v_{i,n_i}$ be the vertices of $T_i$, and
let $N=n_1+\cdots+n_{\ell}$.

\begin{figure}
\begin{center}
\epsfbox{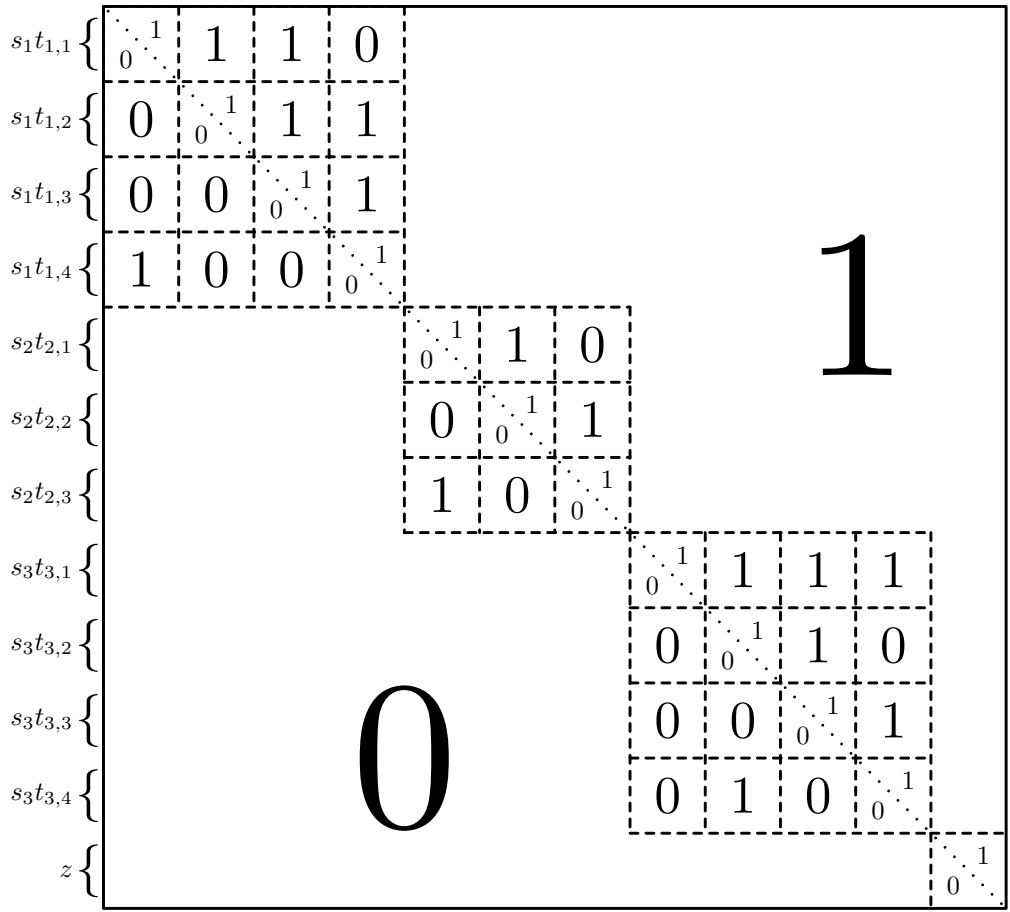}
\end{center}
\caption{The tournamenton $W_4$.
         Note that $\ell=3$ and
	 the tournaments $T_1$, $T_2$ and $T_3$
	 are the tournaments associated with the words $c$, $b$ and $ab$ in Figure~\ref{fig:Lyndon}.
	 The origin of the coordinate system is in the top left corner (in the analogy to adjacency matrices).}
\label{fig:W4}
\end{figure}

We define a family of tournamentons $W_k$ parameterized by $\ell$ variables $s_i$, $i\in [\ell]$, and
$N$ variables $t_{i,j}$, $i\in [\ell]$ and $j\in [n_i]$.
For brevity,
we will refer to the variables $s_i$ as the \emph{$s$-variables} and
to $t_{i,j}$ as the the \emph{$t$-variables}.
Note that the $t$-variables one-to-one correspond to the vertices of the tournaments $T_1,\ldots,T_{\ell}$.
We next define a tournamenton $W_k(s_{[\ell]},t_{1,[n_1]},\ldots,t_{\ell,[n_\ell]})$
for every choice of positive reals $s_i$, $i\in [\ell]$, and
$t_{i,j}$, $i\in [\ell]$ and $j\in [n_i]$, such that
\[\sum_{i\in [\ell]}s_i\sum_{j\in [n_i]}t_{i,j}<1.\]
Informally speaking,
the tournamenton $W_k$ is the ``blow up'' of the tournament $T_1\oplus\cdots\oplus T_{\ell}\oplus\bullet$
with vertices of a tournament $T_i$, $i\in [\ell]$, blown up to an interval of measure $s_it_{i,j}$, $j\in [n_i]$.
To rigorously construct the tournamenton $W_k$,
we first split the interval $[0,1]$ into $N+1$ disjoint intervals:
an interval $I_{i,j}$, $i\in [\ell]$ and $j\in [n_i]$, has measure $s_it_{i,j}$, and
the remaining interval $I_0$ has measure
\[1-\sum_{i\in [\ell]}s_i\sum_{j\in [n_i]}t_{i,j}.\]
If the tournament $T_i$ contains an arc from the vertex $v_{i,j}$ to the vertex $v_{i,j'}$ in $T_i$,
then $W_k$ is equal to one on $I_{i,j}\times I_{i,j'}$ and to zero on $I_{i,j'}\times I_{i,j}$.
For all $1\le i<i'\le\ell$, $j\in [n_i]$ and $j'\in [n_{i'}]$,
the tournamenton $W_k$ is equal to one on $I_{i,j}\times I_{i',j'}$ and to zero on $I_{i',j'}\times I_{i,j}$, and
for all $i\in [\ell]$ and $j\in [n_i]$,
the tournamenton $W_k$ is equal to one on $I_{i,j}\times I_0$ and to zero on $I_0\times I_{i,j}$.
Finally,
if either $x,y\in I_{i,j}$ for some $i\in [\ell]$ and $j\in [n_i]$ or $x,y\in I_0$,
then $W_k(x,y)=1$ if $x<y$, $W_k(x,y)=0$ if $x>y$, and $W_k(x,y)=1/2$ if $x=y$.
The construction of the tournamenton $W_4$ is illustrated in Figure~\ref{fig:W4}.
We remark that
the main reason for including the part corresponding to the interval $I_0$ is that
the expressions for densities of the tournaments $T_1,\ldots,T_{\ell}$ in $W_k$
are homogeneous polynomials in the $s$-variables and the $t$-variables, and
so the analysis of the densities of the tournaments $T_1,\ldots,T_{\ell}$ becomes significantly simpler.

We next consider the function $F:[0,1]^{\ell+N}\to [0,1]^{\ell}$
defined as 
\[F(s_{[\ell]},t_{1,[n_1]},\ldots,t_{\ell,[n_\ell]})_i=t(T_i,W_k(s_{[\ell]},t_{1,[n_1]},\ldots,t_{\ell,[n_\ell]}))\]
for $i\in [\ell]$.
Recall that $s_{[\ell]}$ stands for $s_1,\ldots,s_{\ell}$; $t_{i,[n_i]}$ is used analogously.
We fix $i\in [\ell]$ and inspect the formula \eqref{eq:tdef} that defines $t(T_i,W_k)$.
Observe that if $x\in [0,1]^{V(T_i)}$ and $x_v\in I_0$ for some $v\in V(T_i)$,
then the product in \eqref{eq:tdef} can be non-zero only if $v$ is a sink in $T_i$,
i.e., a vertex with incoming edges only.
Since the tournament $T_i$ has no sink,
it follows that
\begin{equation}
t(T_i,W_k(s_{[\ell]},t_{1,[n_1]},\ldots,t_{\ell,[n_\ell]}))=
  \int\limits_{\left([0,1]\setminus I_0\right)^{V(T_i)}} \prod_{\overrightarrow{uv}\in E(T_i)} W_k(s_{[\ell]},t_{1,[n_1]},\ldots,t_{\ell,[n_\ell]})(x_u,x_v) \dd x_{V(T_i)}.
  \label{eq:setminus}
\end{equation}
Consider $x\in \left([0,1]\setminus I_0\right)^{V(T_i)}$ and
let $f_{x}:V(T_i)\to V(T_1\oplus\cdots\oplus T_{\ell})$ be defined as $f_{x}(v)=v_{j,j'}$ if $x_{v}\in I_{j,j'}$.
Further, let $H(T_i,T_1\oplus\cdots\oplus T_{\ell})$ is the set of all maps $f:V(T_i)\to V(T_1\oplus\cdots\oplus T_{\ell})$ such that
\begin{itemize}
\item the subtournament of $T_1\oplus\cdots\oplus T_{\ell}$ induced by $f^{-1}(v)$ is acyclic for every $v\in V(T_i)$,
      where $f^{-1}(v)$ denotes the preimage of a vertex $v$, and
\item for every edge $uv\in E(T_i)$, either $f(u)=f(v)$ or $f(u)f(v)$ is an edge of $T_1\oplus\cdots\oplus T_{\ell}$.
\end{itemize}
Observe the product in \eqref{eq:setminus} is non-zero for $x\in \left([0,1]\setminus I_0\right)^{V(T_i)}$
only if the function $f_{x}$ belongs to $H(T_i,T_1\oplus\cdots\oplus T_{\ell})$.
We obtain that
\begin{equation}
t(T_i,W_k(s_{[\ell]},t_{1,[n_1]},\ldots,t_{\ell,[n_\ell]}))=
  \sum_{f\in H(T_i,T_1\oplus\cdots\oplus T_{\ell})}
  \prod_{j\in [\ell]}\prod_{j'\in [n_j]}\frac{(s_jt_{j,j'})^{|f^{-1}(v_{j,j'})|}}{|f^{-1}(v_{j,j'})|!}
  \label{eq:H}
\end{equation}
where we set $0!$ to be equal to $1$.
It follows that $t(T_i,W_k(s_{[\ell]},t_{1,[n_1]},\ldots,t_{\ell,[n_\ell]}))$
is a homogeneous polynomial in the $s$-variables and $t$-variables and
the degree of each monomial in the $s$-variables is $n_i$ and in the $t$-variables is also $n_i$;
moreover, the degree of a variable $s_j$ in a monomial
is equal to the sum of the degrees of the variables $t_{j,1},\ldots,t_{j,n_j}$ in the monomial.

Our aim is to analyze the Jacobian matrix $\JJ$ of the function $F$ with respect to the $s$-variables,
i.e.,
\[\JJ_{i,j}=\frac{\partial }{\partial s_j}t(T_i,W_k(s_{[\ell]},t_{1,[n_1]},\ldots,t_{\ell,[n_\ell]}))\]
for $i,j\in [\ell]$.
Note that $\JJ_{i,j}$ is a homogeneous polynomial of degree $n_i-1$ in the $s$-variables and of degree $n_i$ in the $t$-variables.
It follows that the determinant $\det(\JJ)$ is a homogeneous polynomial of degree $N-\ell$ in the $s$-variables and
of degree $N$ in the $t$-variables (note that the total number of $t$-variables is $N$).
We will show that the polynomial $\det(\JJ)$ is non-zero,
which will be established by identifying a monomial of $\det(\JJ)$ with a non-zero coefficient.
As the first step towards this goal, we prove the next lemma.

\begin{lemma}
\label{lm:lem7}
Let $i\in [\ell]$,
let $X$ be any set containing $n_i$ distinct $t$-variables, and
let $Y$ be the set containing the $n_i$ vertices of the tournaments $T_1,\ldots,T_\ell$ that
correspond to the variables contained in $X$.
The polynomial $t(T_i,W_k(s_{[\ell]},t_{1,[n_1]},\ldots,t_{\ell,[n_\ell]}))$
has a monomial including the product of the variables from $X$ (with a non-zero coefficient) if and only if
the vertices contained in $Y$ induce a copy of $T_i$ in the tournament $T_1\oplus\cdots\oplus T_{\ell}$.
\end{lemma}

\begin{proof}
Fix $i\in [\ell]$ and the sets $X$ and $Y$.
We inspect the sum in \eqref{eq:H} that
defines $t(T_i,W_k(s_{[\ell]},t_{1,[n_1]},\ldots,t_{\ell,[n_\ell]}))$.
A monomial in the sum in \eqref{eq:H} contains $n_i$ distinct $t$-variables
if and only if the function $f\in H(T_i,T_1\oplus\cdots\oplus T_{\ell})$ is injective.
It follows that any monomial including the product of the variables from $X$
arises from an injective function $f\in H(T_i,T_1\oplus\cdots\oplus T_{\ell})$ that
maps the vertices of $T_i$ to the set $Y$.
Since the injective functions contained in $H(T_i,T_1\oplus\cdots\oplus T_{\ell})$
are exactly injective homomorphisms from $T_i$ to $T_1\oplus\cdots\oplus T_{\ell}$,
the polynomial $t(T_i,W_k(s_{[\ell]},t_{1,[n_1]},\ldots,t_{\ell,[n_\ell]}))$
has a monomial including the product of the variables from $X$ with a non-zero coefficient if and only if
the set $Y$ induces a copy of $T_i$ in the tournament $T_1\oplus\cdots\oplus T_{\ell}$.
\end{proof}

To show that the polynomial $\det(\JJ)$ is non-zero,
we will also need the following lemma.

\begin{lemma}
\label{lm:decomp}
Let $Y_1,\ldots,Y_{\ell}$ be disjoint subsets of the vertices of the tournament $T_1\oplus\cdots\oplus T_{\ell}$.
If the subtournament induced by $Y_i$ in $T_1\oplus\cdots\oplus T_{\ell}$ is isomorphic to $T_i$ for every $i\in [\ell]$,
then $Y_i=V(T_i)$.
\end{lemma}

\begin{proof}
Let $m_i$ be the number of strongly connected components of the tournament $T_i$, $i\in [\ell]$.
Since the number of strongly connected components of the tournament $T_1\oplus\cdots\oplus T_{\ell}$
is $m_1+\cdots+m_{\ell}$, and
the vertices of a strongly connected component of the copy of $T_i$ induced by $Y_i$ in $T_1\oplus\cdots\oplus T_{\ell}$
must be contained in the same strongly connected component of $T_1\oplus\cdots\oplus T_{\ell}$,
the vertices of a strongly connected component of the copy of $T_i$ induced by $Y_i$
actually form a strongly connected component of $T_1\oplus\cdots\oplus T_{\ell}$.
For $i\in [\ell]$,
let $Z_i$ be the set of strongly connected components of $T_1\oplus\cdots\oplus T_{\ell}$
such that $Y_i$ is the union of the vertex sets of the components contained in $Z_i$.
Observe that the subword of the word $w(T_1\oplus\cdots\oplus T_{\ell})$
induced by the letters corresponding to the components contained in $Z_i$
is the word $w(T_i)$.
Since the word $w(T_1)\cdots w(T_{\ell})$
is the lexicographically largest constituent in the shuffle product $w(T_1)\otimes_S\cdots\otimes_S w(T_{\ell})$ and
its coefficient in the shuffle product is equal to one by Proposition~\ref{prop:le1},
it follows that there is a unique way to split the word $w(T_1)\cdots w(T_{\ell})$
into disjoint subwords $w(T_1),\ldots,w(T_{\ell})$.
Hence, $Z_i$ are the components of $T_i$ for every $i\in [\ell]$ (otherwise,
the way to split the word $w(T_1)\cdots w(T_{\ell})$ into disjoint subwords $w(T_1),\ldots,w(T_{\ell})$
would not be unique).
It follows that $Y_i=V(T_i)$ for every $i\in [\ell]$.
\end{proof}

We are now ready to prove the main theorem of this section.
While the value $k\ge 3$ is fixed throughout this section,
we include it in the statement of the theorem for easier reference to the theorem outside this section.

\begin{theorem}
\label{thm:lowth}
For every $k\ge 3$,
the determinant $\det(\JJ)$ is a non-zero polynomial.
In particular,
there exists a choice of the $s$-variables and the $t$-variables such that
the value of the determinant $\det(\JJ)$ is non-zero, and
so there exist $x_0\in [0,1]^{\TT^L_k}$ and $\varepsilon>0$ such that
for every $x\in B_{\varepsilon}(x_0)\subseteq [0,1]^{\TT^L_k}$,
there exists a choice of the $s$-variables and the $t$-variables such that
\[t\left(W_k(s_{[\ell]},t_{1,[n_1]},\ldots,t_{\ell,[n_\ell]})\right)_{T\in\TT^L_k}=x.\]
\end{theorem}

\begin{proof}
Recall that $\ell$ is the number of non-trivial Lyndon tournaments with at most $k$ vertices.
We first prove that $\det(\JJ)$ is a non-zero polynomial.
To do so, we show that the determinant $\det(\JJ)$ contains the monomial
\[\prod_{i\in [\ell]}s_i^{n_i-1}\prod_{j\in [n_i]}t_{i,j}\]
with a non-zero coefficient.
Recall that
\[\det(\JJ)=\sum_{\pi\in S_\ell}\sgn(\pi)\prod_{i\in [\ell]}\JJ_{i,\pi(i)}\]
is a homogeneous polynomial of degree $N-\ell$ in the $s$-variables and
of degree $N$ in the $t$-variables (note that the number of $t$-variables is $N$).
Let $\pi$ be a permutation such that $\prod\JJ_{i,\pi(i)}$
has a monomial containing each of the $N$ $t$-variables.
This monomial is a product of some monomials in $\JJ_{i,\pi(i)}$ and
let $X_i$ be the $n_i$ $t$-variables contained in the corresponding monomial in $\JJ_{i,\pi(i)}$.
Note that the sets $X_1,\ldots,X_{\ell}$ partition the set of $t$-variables.
Further,
let $Y_i$ be the vertices of the tournament $T_1\oplus\cdots\oplus T_{\ell}$ 
corresponding to the $t$-variables contained in $X_i$, $i\in [\ell]$.
Since the sets $X_1,\ldots,X_{\ell}$ are disjoint,
the sets $Y_1,\ldots,Y_{\ell}$ are also disjoint and so partition the vertex set of $T_1\oplus\cdots\oplus T_{\ell}$.

Since $\JJ_{i,\pi(i)}$
is the partial derivative of the polynomial $t(T_i,W_k(s_{[\ell]},t_{1,[n_1]},\ldots,t_{\ell,[n_\ell]}))$ with respect to $s_{\pi(i)}$,
the polynomial $t(T_i,W_k(s_{[\ell]},t_{1,[n_1]},\ldots,t_{\ell,[n_\ell]}))$
has a monomial including the product of the variables from $X_i$.
Lemma~\ref{lm:lem7} implies that the set $Y_i$ induces a copy of $T_i$ in the tournament $T_1\oplus\cdots\oplus T_{\ell}$.
Since the sets $Y_1,\ldots,Y_{\ell}$ are disjoint subsets of the vertices of the tournament $T_1\oplus\cdots\oplus T_{\ell}$,
Lemma~\ref{lm:decomp} yields that $Y_i=V(T_i)$.
Hence, $X_i=\{t_{i,1},\ldots,t_{i,n_i}\}$ and
the monomial of $t(T_i,W_k(s_{[\ell]},t_{1,[n_1]},\ldots,t_{\ell,[n_\ell]}))$ that
yields the monomial of $\JJ_{i,\pi(i)}$ containing the product of the variables from $T_i$
is $s_i^{n_i}t_{i,1}\cdots t_{i,n_i}$.
It follows that $\pi(i)=i$ (otherwise, the partial derivative of the monomial with respect to $s_{\pi(i)}$ would be zero).
We conclude that the polynomial $\det(\JJ)$ has a single monomial containing all $t$-variables,
this monomial is 
\[\prod_{i\in [\ell]}s_i^{n_i-1}\prod_{j\in [n_i]}t_{i,j},\]
and the coefficient at this monomial is the product of the sizes of the automorphism groups of $T_1,\ldots,T_{\ell}$.
In particular, the coefficient is non-zero.

It remains to verify
the existence of $x_0\in [0,1]^{\TT^L_k}$ and $\varepsilon>0$ with the properties given in the statement of the theorem.
Since $\det(\JJ)$ is a non-zero polynomial,
there exist positive reals $s_i$, $i\in [\ell]$, and $t_{i,j}$, $i\in [\ell]$ and $j\in [n_i]$, such that
$\det(\JJ)$ is non-zero and
\[\sum_{i\in [\ell]}s_i\sum_{j\in [n_i]}t_{i,j}<1.\]
Next define a function $G:[0,1]^{\ell}\to [0,1]^{\ell}$ as
\[G(z_1,\ldots,z_{\ell})=F(z_1,\ldots,z_{\ell},t_{1,[n_1]},\ldots,t_{\ell,[n_\ell]})\]
where $F:[0,1]^{\ell+N}\to [0,1]^{\ell}$ is the function defined in the beginning of this section, and
set $x_0=F(s_{[\ell]},t_{1,[n_1]},\ldots,t_{\ell,[n_\ell]})$.
Note that $(x_0)_{T_i}$
is the density of the tournament $T_i$ in the tournamenton $W_k(s_{[\ell]},t_{1,[n_1]},\ldots,t_{\ell,[n_\ell]})$.
Since the Jacobian matrix of $G$ for $z_{[\ell]}=s_{[\ell]}$ is exactly the matrix $\JJ$ and
the determinant of the Jacobian matrix of $G$ is a polynomial in $z_{[\ell]}$ (as the function $G$
itself is a polynomial in $z_{[\ell]}$),
there exists $\delta>0$ such that
the Jacobian matrix of $G$ for every $z\in B_{\delta}(s_{[\ell]})$ is invertible and
\[\sum_{i\in [\ell]}z_i\sum_{j\in [n_i]}t_{i,j}<1\]
for every $z\in B_{\delta}(s_{[\ell]})$,
i.e., the tournamenton $W_k(z_{[\ell]},t_{1,[n_1]},\ldots,t_{\ell,[n_\ell]})$ is well-defined.
The Inverse Function Theorem implies the existence of $\varepsilon>0$ that
for every $x\in B_{\varepsilon}(x_0)$,
there exists (unique) $z\in B_{\delta}(s_{[\ell]})$ such that $G(z)=x$.
Since the density of the tournament $T_i$ in the tournamenton $W_k(z_{[\ell]},t_{1,[n_1]},\ldots,t_{\ell,[n_\ell]})$ is $G(z)_i$,
the statement of the theorem now follows.
\end{proof}

\section*{Acknowledgement}

The authors would like to thank the anonymous reviewer for their detailed comments on the arguments presented in the paper and
particularly for pointing out that Figure~\ref{fig:W4} in the originally submitted version contained a wrong tournamenton.

\bibliographystyle{bibstyle}
\bibliography{tourn-dim}

\end{document}